\documentclass[reqno]{amsart}

\usepackage[centering]{geometry}

\usepackage{graphicx}

\usepackage{mathtools}
\usepackage{mathrsfs}
\usepackage{amsmath}
\usepackage{bm}
\usepackage{tikz-cd}
\usepackage{amsthm}
\usepackage{enumitem}
\definecolor{calpolypomonagreen}{rgb}{0.12, 0.3, 0.17}
	\definecolor{asparagus}{rgb}{0.53, 0.66, 0.42}
	\definecolor{applegreen}{rgb}{0.55, 0.71, 0.0}
		\definecolor{darkpastelgreen}{rgb}{0.20, 0.70, 0.24}
		\definecolor{amaranth}{rgb}{0.85, 0.17, 0.31}
			\definecolor{darkpastelred}{rgb}{0.76, 0.23, 0.13}
\usepackage[colorlinks=true, linkcolor = darkpastelred, citecolor = darkpastelgreen, urlcolor = ]{hyperref}

\usepackage{accents}

\newtheorem{theorem}{Theorem}[section]
\newtheorem{lemma}[theorem]{Lemma}

\newtheorem{corollary}[theorem]{Corollary}

\newtheorem{thm}[theorem]{Theorem}
\newtheorem{cor}[theorem]{Corollary}
\theoremstyle{definition}
\newtheorem{remark}[theorem]{Remark}

\newtheorem{definition}[theorem]{Definition}
\newtheorem*{definition*}{Definition}
\newtheorem{example}[theorem]{Example}

\theoremstyle{definition}
\newtheorem{notationlist}{Notation list}

\usepackage{amssymb}
\renewcommand{\leq}{\leqslant}
\renewcommand{\geq}{\geqslant}
\newcommand{\eps}{\epsilon}
\newcommand{\ca}{\mathcal}
\newcommand{\wt}{\widetilde}

\usepackage{todonotes}
\newcommand{\daniele}[1]{\todo[color=green!40]{Daniele: #1}{}}

\def\N{\mathbf N}

\def\dim{\mathrm{dim}}

\def\GL{\mathrm{GL}}
\def\Gal{\mathrm{Gal}}

\def\Gm{\Gamma}

\def\urad{\mathrm R_u}
\def\Cay{\mathrm{Cay}}
\def\det{\mathrm{det}}

\def\SL{\mathrm{SL}}

\def\P{\mathbf P}

\def\Aut{\text{Aut}}

\def\F{\mathbf F}
\def\Q{\mathbf Q}

\def\Zint{\mathbf Z}
\def\Hom{\mathrm{Hom}}
\def\ca{\mathcal}
\def\p{\mathfrak p}
\def\a{\mathfrak a}
\def\b{\mathfrak b}
\def\q{\mathfrak q}

\def\ul{\underline}

\newcommand{\gen}[1]{\ensuremath{\langle #1\rangle}}

\usepackage{accents}
\newcommand\thickbar[1]{\accentset{\rule{.7em}{.8pt}}{#1}}

\makeatletter
\def\blfootnote{\xdef\@thefnmark{}\@footnotetext}
\makeatother

\begin{document}
\title{Hilbert's irreducibility theorem via random walks}
\date{\today}

\author{Lior Bary-Soroker}
\address{Lior Bary-Soroker, Raymond and Beverly Sackler School of Mathematical Sciences, Tel Aviv University,
Tel Aviv 69978, Israel}
\email{barylior@tauex.tau.ac.il}

\author{Daniele Garzoni}
\address{Daniele Garzoni, Raymond and Beverly Sackler School of Mathematical Sciences, Tel Aviv University, Tel Aviv 69978, Israel}
\email{danieleg@mail.tau.ac.il}

\maketitle

\begin{abstract}
 Let $G$ be a connected linear algebraic group over a number field $K$, let $\Gm$ be a finitely generated Zariski dense subgroup of $G(K)$ and let $Z\subseteq G(K)$ be a thin set, in the sense of Serre. We prove that, if $G/\urad(G)$ is semisimple and  $Z$ satisfies certain necessary conditions, then a long random walk on a Cayley graph of $\Gm$ hits elements of $Z$ with negligible probability. We deduce corollaries to Galois covers, characteristic polynomials, and fixed points in group actions. We also prove analogous results 
 in the case where $K$ is a global function field.
 \end{abstract}

\section{Introduction} 

\subsection{Hilbert's irreducibility theorem}
This paper concerns generalizations of the celebrated Hilbert's Irreducibility Theorem (HIT). In polynomial terms, HIT asserts that, for every $n\geq 1$, if $f(t_1, \ldots, t_n,X)\in \Q(t_1, ..., t_n)[X]$ is irreducible, then there exists an \emph{irreducible specialization}, that is to say, $(\ul{t_1}, \ldots, \ul{t_n})\in \Q^n$ such that $f(\ul{t_1}, \ldots, \ul{t_n}, X)\in \Q[X]$ remains irreducible. Moreover, the Galois group of $f(\ul{t_1}, \ldots, \ul{t_n}, X)$ over $\Q$ is isomorphic to the Galois group of $f(t_1, \ldots, t_n,X)$ over $\Q(t_1, \ldots, t_n)$. Hilbert's original motivation came from the inverse Galois problem: In order to realize a finite group as a Galois group over $\Q$, given HIT it is sufficient to realize it over $\Q(t_1, \ldots, t_n)$, where one has more tools.

HIT found numerous applications in number theory, algebraic geometry, and algebra, see \cite{serre_mordell_weyl, serre_topics_galois, fried_jarden3, malle_matzat_igt}. 

HIT has an equivalent formulation in terms of rational points, an approach with traces going back at least to Hilbert's works on the inverse Galois problem. It uses the notion of  \textit{thin sets}, introduced by Serre \cite{serre_mordell_weyl} and Colliot-Th\'{e}l\`{e}ne--Sansuc \cite{colliotthelene_sansuc}.  
In this paper, a $K$-variety is a separated $K$-scheme of finite type. 

\begin{definition*}
\begin{enumerate}
    \item 
    Let $K$ be a field and $V$ be an irreducible $K$-variety.
A subset $Z$ of $V(K)$ is called \textbf{thin} if 
\begin{equation}
\label{eq_thin_sets}
    Z\subseteq C(K)\cup \bigcup_{i=1}^t\pi_{i}(V_i(K)),
\end{equation}
where $C$ is a proper closed subset of $V$, $t\in \Zint_{\geq 0}$, and for every $i=1, \ldots, t$, $V_i$ is an irreducible $K$-variety, $\dim(V_i) = \dim(V)$, and $\pi_i\colon V_i\to V$ is a dominant separable morphism with $\deg(\pi_i)\geq 2$.
    \item We say that $V$ has the \textbf{Hilbert property} (HP) if $V(K)$ is not thin.
\end{enumerate}
\end{definition*}

HIT is equivalent to the statement that $\mathbb{A}^n_\Q$ has the HP for every $n\geq 1$. Varieties having the HP satisfy an irreducibility specialization property, see \cite[Proposition 3.3.5]{serre_topics_galois}. Every connected linear algebraic group over a number field has the HP \cite{sansuc1981groupe, colliotthelene_sansuc, baysoroker_fehn_petersen}.
See \cite{corvaja_zannier_fundamental_group}, where the HP is proven 
for a geometrically non-rational variety for the first time, and see its extensions in 
\cite{demeio2020}.

For many $V$ satisfying the HP, not only $V(K)$ is not thin, but a thin set $Z$ is small in $V(K)$ in much stronger senses.  For example, if $V=\mathbb{A}^n_\Q$, then $H= V(\Q)\smallsetminus Z$  is dense in the real topology, it contains arithmetic progressions, and $H$ has asymptotic density $1$, with respect to counting either integral or rational points by height.

For a connected linear algebraic group $V=G$ over a number field $K$,  Liu \cite{FeiLiu} and Corvaja \cite{corvaja2007rational}, based on Zannier \cite{zannier2010hilbert} and Ferretti--Zannier \cite{ferretti2007equations}, proved that thin sets $Z$ are small with respect to the group structure. More precisely, let $\Gm$ be a Zariski dense subgroup of $G(K)$
and let $Z$ be a thin set of $G(K)$ as in \eqref{eq_thin_sets}. For every $i$, assume that $V_i$ is normal and $\pi_i\colon V_i\to G$ is finite and \textit{ramified}. Then, $\Gm\nsubseteq Z$.


The ramification assumption is essential for the conclusion. Indeed, if $\pi\colon G_1\to G$ is a nontrivial isogeny of algebraic groups; in particular, $\pi$ is unramified, then $\Gm:=\pi(G_1(K))$ is a Zariski dense subgroup of $G(K)$, and also a thin set. We remark that the connection between the HP and ramification is made explicit in \cite{corvaja_zannier_fundamental_group}.

The objective of this paper is to obtain a quantitative version of this result. Loosely speaking, we will show that, under suitable conditions on $G$, ``most'' points of $\Gm$ do not belong to $Z$. We will measure this by performing long random walks on a Cayley graph of $\Gm$. This direction is inspired by the works of Rivin \cite{rivin2008walks}, Jouve--Kowalski--Zywina \cite{kowalski_jouve_zywina}, Lubotzky--Meiri \cite{lubotzky_meiri_powers}, and Lubotzky--Rosenzweig \cite{lubotzky_rosenzweig}.


\subsection{Random walks}
\label{sub_intro_random_walks}
We fix notation for this subsection.

\begin{notationlist}
\label{notation_1} 
\begin{itemize} 
    \item[$\diamond$] $K$ is a number field.
    \item[$\diamond$] $G$ is a connected linear algebraic $K$-group such that $G/\text R_u(G)$ is either trivial or semisimple.
    \item[$\diamond$] $\Gamma$ is a finitely generated Zariski dense subgroup of $G(K)$.
    \item[$\diamond$] $A$ is a finite symmetric generating multiset of $\Gm$, with $1\in A$.
    \item[$\diamond$] $\omega_n$ is the $n$-th step of a random walk on the Cayley graph $\Cay(\Gm,A)$, starting at the identity.
    \item[$\diamond$] For every $G$, we define a function $\Xi\colon \mathbf R_{>0} \times \mathbf Z_{\geq 1} \to \mathbf R_{>0}$, as follows:
\[
   \Xi(C,n) = \begin{cases}
    e^{-n/C} &\mbox{if $G$ is semisimple}\\
     Cn^{-1/(10\,\dim \, G)} &\mbox{otherwise.} 
     \end{cases}
\]

\end{itemize}

\end{notationlist}


Let us clarify the notation: 
By $\urad(G)$, we  denote the unipotent radical of $G$. 
By $A$ symmetric, we mean that  $g$ and $g^{-1}$ appear in $A$ with the same multiplicity, for every $g\in \Gm$. The Cayley graph $\Cay(\Gm,A)$ attached to $\Gm$ and $A$ has  $\Gm$ as the set of vertices and $\{ (g,ga) :  a\in A, g\in \Gm\}$ as  the multiset of edges. Then,  $\omega_n$ is the $n$-th step of the uniform random walk on $\Cay(\Gm,A)$, starting at the identity. 


Our main  theorem asserts that, under our assumptions on $G$, the random walk $\omega_n$ is asymptotically almost never in a thin set:

\begin{thm}
\label{t_main_new}
Let $Z$ be a thin subset of $G(K)$ as in \eqref{eq_thin_sets} with the additional assumptions that  $V_i$ is normal and $\pi_i\colon V_i\to G$ is finite and ramified, for each $i$. Then, 
there exists $C>0$ such that 
\[
    \P(\omega_n \in Z) \leq \Xi(C,n) .
\]
In particular, $\displaystyle\lim_{n\to \infty}\P(\omega_n \in Z) = 0 $.
\end{thm}




If $G$ is semisimple, Theorem~\ref{t_main_new} gives exponential decay. When $G$ is not semisimple,  one cannot get exponential decay. Indeed, take $G=\mathbb A^1$, $\Gm=\Zint$,  $A=\{0,1,-1\}$, and consider the thin set $Z=\{0\}$. Then, $\omega_n$ is a classical one dimensional random walk, hence
\begin{equation}\label{r_intro_example}
    \P(\omega_n \in Z)=\P(\omega_n =0) \geq c n^{-1/2}.
\end{equation}
for some $c>0$.

If $G$ is not semisimple, then $C$ in Theorem~\ref{t_main_new}  depends on $G$, $K$, $\Gm$, $|A|$, and on $Z$. In fact, the dependence on $Z$ is via the \emph{complexity}, see Definition~\ref{def_complexity}. If $G$ is semisimple, $C$
depends also on the spectral gap of the Cayley graphs induced on the congruent quotients of $\Gm$; see \S~\ref{sec_simply_connected} for further details. In all the theorems of the introduction, $C$ will depend on these parameters.

As in Liu's result, the assumption that the morphisms $\pi_i$ are ramified is crucial: If $\pi\colon G_1 \to G$ is a nontrivial isogeny of algebraic groups, then $\P(\omega_n\in \pi(G_1(K))) \geq c>0$. See Remark~\ref{r_necessary_ramified} below for further details.

The assumption that the varieties $V_i$ are normal and that the morphisms are finite is technical and can be removed. See Remark~\ref{r_not_normal} and Lemma~\ref{l_pullback}.

It seems that Theorem~\ref{t_main_new} was known to experts in the case where $G$ is semisimple and simply connected (\cite[\S~6.2]{kowalski2010sieve}). 
We note that the estimates in \emph{loc.\ cit.} are not valid for semisimple non-simply connected groups, as the example of a nontrivial isogeny shows.

We apply Theorem~\ref{t_main_new} to get irreducibility, in the spirit of   the classical HIT formulation.

\begin{cor}
\label{c_main_irreducibility}
Let $V$ be a normal irreducible $K$-variety, and $\pi\colon V\to G$ be a finite surjective morphism. Assume that every non-scalar subcover of $\pi$ is ramified. 
Then, there exists $C>0$ such that 
\[
    \P(\pi^{-1}(\omega_n) \mbox{ is not integral\,}) \leq \Xi(C,n).
\]
\end{cor}

The original motivation of Hilbert came from Galois theory, where he used HIT to specialize Galois groups from $\Q(t)$ to $\Q$. The next application is in this spirit. 
If $\pi\colon V\to G$ is a generically Galois cover of normal irreducible $K$-varieties with group $\mathcal{G}=\Gal(K(V)/K(G))$, then every $g\in G(K)$ comes with the \textit{specialized Galois group}, denoted $\mathcal{G}_g$, which  is the automorphism group of the residue field extension at a point $x\in V$ lying above $g$; this group is isomorphic to the decomposition group modulo the inertia group. Hence, the specialized Galois group is always a sub-quotient of $\mathcal{G}$. We show that typically it equals $\mathcal{G}$:

\begin{cor}
\label{c_main_new}
Let $V$ be a normal irreducible $K$-variety, and $\pi\colon V\to G$ be a finite surjective morphism, which is generically Galois with group $\mathcal{G}$. Assume that every nontrivial geometrically irreducible subcover of $\pi$ is ramified. 
Then, there exists $C>0$ such that 
\[
    \P(\mathcal{G}_{\omega_n} \neq \mathcal{G})\leq \Xi(C,n).
\]
\end{cor}

A  different sense for smallness of a thin set $Z$ in an algebraic group $G$ is the following: Take $\Gm \subseteq G(K)$ a finitely generated Zariski dense subgroup and prove that there exists a finite index coset that is disjoint from $Z$. As in our case, a ramification condition is necessary. Schinzel \cite{Schinzel} proved it for the additive group $\mathbb G_a$, Zannier \cite{zannier2010hilbert} for a torus or a power of a non-CM elliptic curve.  Corvaja, Demeio, Javanpeykar, Lombardo, and  Zannier  \cite{corvaja_zannier_demeio} generalized the latter to any abelian variety, and Liu \cite{FeiLiu} and Corvaja \cite{corvaja2007rational} generalized it to any connected linear algebraic group.  



\subsection{Characteristic polynomial}
\label{sub_intro_char_poly}
We are still under Notation List~\ref{notation_1}. The typical Galois group of the characteristic polynomial of $\omega_n$ is well studied in the literature, see \cite{rivin2008walks,kowalski_jouve_zywina,lubotzky_rosenzweig}. 

Our objective in this subsection is to show how these type of results follow from the general result on generic Galois groups (Corollary~\ref{c_main_new}). The advantage of our approach is that we do not use anything on the characteristic polynomial, except for the computation of the generic Galois group by Prasad and Rapinchuk. The original proofs  use  more specific information, see for instance \cite[\S~4 and Proposition 4.1]{kowalski_jouve_zywina}. 

Fix an embedding $\rho\colon G \hookrightarrow \GL_N$. Then, we have the generic  characteristic polynomial $\chi_{G,\rho}  \in K(G)[X]$ such that $\chi_{G,\rho}|_{g} = \det (X-g^{\rho})$ for every $g\in G(K)$. Let $V$ be the normalization of $G$ in the splitting field $F$ of $\chi_{G,\rho}$ over $K(G)$, and let $\pi\colon V\to G$ be the corresponding finite morphism, which is generically Galois.

It turns out that $F$ is independent of $\rho$. It follows from a result of Prasad and Rapinchuk \cite{prasad2003existence} that $\Gal(F/K(G))$ is isomorphic to a certain extension of the Weyl group $W(G)$ of $G$, that we denote by $\Pi(G)$. For instance, if $G$ is split, then $\Pi(G)=W(G)$. See \S~\ref{sub_characteristic_polynomial} for the precise description of $\Pi(G)$.

It is easy to see that, in this problem, one can replace $G$ by $G/\urad(G)$. We will deduce from Corollary~\ref{c_main_new} the following:



\begin{cor}
\label{c_main_char_pol}
Assume that  $G/\urad(G)$ is either trivial or semisimple, and put $\chi_n = \det(X-\omega_n^\rho)$.
Then, there exists $C>0$ such that 
\[
    \P(\Gal(\chi_n/K) \not\cong \Pi(G))\leq e^{-n/C}.
\] 
\end{cor}

A more general result  was proved by Lubotzky and Rosenzweig \cite{lubotzky_rosenzweig}, who also considered non-connected groups.


In the case of arithmetic subgroups of
simple simply connected groups over a number field, Gorodnik and Nevo \cite{gorodnik_nevo_char_polynomial} considered the typical characteristic polynomial of elements by counting in archimedean balls, instead of using random walks. This is a substantially different problem. 
We remark that the methods in \cite{gorodnik_nevo_char_polynomial} can be used to deal with thin sets in the archimedean model.

\subsection{Fixed points}
We are still under Notation list~\ref{notation_1}. In this subsection, we deal with actions of $G$ on varieties $X$. We show that typically $\omega_n$ has no rational fixed points, unless generically every element has a fixed point.  Our result  is the random walk analogue of a theorem of Corvaja \cite{corvaja2007rational}. 

Let us be more precise. Assume that $G$ acts, as a group scheme, on a $K$-variety $X$. We can then consider the variety of fixed points  
\[
    V=\{(x,g) \mid xg=g\}\subseteq X\times G.
\]
It is equipped with the natural projection morphism $\pi\colon V \to G$. Let $\thickbar K$ be an algebraic closure of $K$. In order to apply Theorem~\ref{t_main_new}, we need the existence of a Zariski dense subset of $G(\thickbar K)$ consisting of elements fixing finitely many points on $X(\thickbar K)$. For technical reasons, we will also require that, letting $p\colon \wt G \to G$ be the universal cover of $G$ (which exists since $G/\urad(G)$ is semisimple), we have that $p^{-1}(\Gm)\leq \wt G(K)$. In applications, this extra assumption can be obtained by replacing $K$ by a finite extension.
We write $\theta\colon U\dashrightarrow W$ for  rational maps, i.e., defined on an open dense $U_{0}\subseteq U$. 

\begin{cor}
\label{c_main_fixed_points}
Assume that, with notation as above, $p^{-1}(\Gm)\leq \wt G(K)$, and assume that there exists a Zariski dense subset of $G(\thickbar K)$ consisting of elements fixing finitely many points on $X(\thickbar K)$. Then one of the following holds: 
\begin{enumerate}
    \item[(i)] There exists a  rational map $\theta\colon G \dashrightarrow X$ such that
     $\theta(g)g = \theta(g)$ for every $g$ in the domain of $\theta$. Moreover, if $X$ is projective, every element of $G(K)$ fixes a point of $X(K)$.
    \item[(ii)] There exists $C>0$ such that 
    \[
        \P(\omega_n \mbox{ has a fixed point in $X(K)$})\leq \Xi(C,n).
    \]
\end{enumerate}
\end{cor}

We have, therefore, a dichotomy: Either generically every element of $G$ fixes a point on $X$, or almost surely a long random walk will hit elements with no rational fixed points. 

To have an example, 
one may take $X$ to be a flag variety (Corollary~\ref{c_flags}). For more  examples, we refer the reader to \cite{corvaja2007rational}.

\subsection{Function fields} \label{sub_intro_function_fields} 
The goal of this subsection is to extend our results to  global fields of positive characteristic. The approach we take over number fields works for general global fields. However, the proof of Theorem~\ref{t_main_new} is based on the ``strong approximation theorem'', which is more delicate in positive characteristic (see, e.g., \cite[pp. 416--417]{lubotzky_segal}). To overcome this obstacle, we add assumptions on $G$ and $\Gm$, in such a way that  the setting remains quite general.

We start by modifying Notation list~\ref{notation_1}: $K$ denotes any global function field. We assume that  
$G$ is  simply connected and semisimple. (See Remark~\ref{rem_function_general} for a more general setting to which our results apply.) In this case, $G=\prod_{i=1}^t G_i$, where $G_i$ is simply connected and $K$-simple for each $i=1,\ldots ,t$. 
 
We consider more specific $\Gamma$. 
Let $S$ be a finite set of places of $K$, let $\ca O_S$ be the ring of $S$-integers of $K$ and, for $\nu\in S$, let $K_\nu$ be the completion of $K$ with respect to $\nu$. Assume that $\prod_{\nu \in S}G_i(K_\nu)$ is non-compact for every $i=1\ldots, t$; in particular, $S$ is non-empty. We then take 
\[
    \Gm=G(\ca O_S)
\]
to be the corresponding $S$-arithmetic subgroup, and we assume that $\Gm$ is finitely generated. 
To have a concrete example in mind, we note that for the group $G=\SL_N$ with $N\geq 3$, all these assumptions are satisfied for every non-empty $S$ and for every $K$. 


\begin{thm}
\label{t_main_new_function_fields}
Let $Z$ be a thin set of $G(K)$. Then, there exists $C>0$ such that
 \[
\P(\omega_n \in Z) \leq C n^{-1/(10\, \dim \, G)}.
\]
\end{thm}

We extend  Corollary~\ref{c_main_new} to function fields.
\begin{cor}
\label{c_function_fields_galois}
    Let $V$ be a normal irreducible $K$-variety, and 
     $\pi\colon V\to G$ be a finite surjective morphism,
   which is generically Galois with group $\mathcal{G}$. Then, there exists $C>0$ such that 
\[
    \P(\mathcal{G}_{\omega_n} \neq \mathcal{G})\leq C n^{-1/(10\, \dim \, G)}.
\]
\end{cor}

We also obtain a result on characteristic polynomials over function fields. 

\begin{cor}
\label{c_main_function_fields_2}
Fix an embedding $\rho\colon G\hookrightarrow \GL_N$. Let $\Pi(G)$ be the extension of the Weyl group of $G$, as explained in \S~\ref{sub_intro_char_poly}, and put $\chi_n = \det(X-\omega_n^\rho)$. Then, there exists $C>0$ such that  
\[
    \P( \chi_n \mbox{ has an inseparable irreducible factor or }
    \Gal(\chi_n/K)\not\cong\Pi(G)) \leq C n^{-1/(10\, \dim \, G)}.
\]
\end{cor}

We remark that there exists a finite place $\nu$ such that, if $\nu \in S$, then Theorem~\ref{t_main_new_function_fields}, Corollary~\ref{c_function_fields_galois} and Corollary~\ref{c_main_function_fields_2} hold with an exponential bound $e^{-n/C}$. See Remark~\ref{r_lubotzky_zuk}.



\subsection{Acknowledgements}
The authors are grateful to Arno Fehm and Sebastian Petersen for stimulating discussions. They also thank Alex Lubotzky  and Alexei Entin for an explanation on strong approximation and on Bertini theorem, respectively. Finally, they thank Pietro Corvaja for pointing out Fei Liu's paper. 

Both authors were partially supported by a grant of the Israel Science Foundation No. 702/19, and the second named author has received funding from the European Research Council (ERC) under the European Union’s Horizon 2020 research and innovation programme (grant agreement No. 850956).

\section{Simply connected groups: Large sieve}
\label{sec_simply_connected}
In this section we will prove Theorem~\ref{t_main_new} in the case where $G$ is simply connected, and we will prove Theorem~\ref{t_main_new_function_fields}. We introduce the following technical notion. 
\begin{definition}
\label{def_complexity}
Let $K$ be a field, $V$ be an irreducible $K$-variety, and $Z$ be a thin set of $V(K)$. We define the \textit{complexity} of $Z$ to be the minimum of 
\[
    \max\{ t, \deg(V_i), \deg(\pi_i), N \mid i=0,\ldots, t\},
\]
running over all $t\geq 0$, $N\geq 1$, subvarieties $V_0, \ldots, V_t\subseteq \mathbb{P}_K^{N}$, and morphisms $\pi_i\colon V_i\to V$ subject to the following conditions: $\dim(V_0)<\dim(V)$; for every $i\geq 1$, $V_i$ is irreducible, $\dim(V_i) = \dim(V)$, $\pi_i$ is dominant separable with $\deg(\pi_i) \geq 2$; $Z\subseteq \bigcup_{i=0}^t \pi_i(V_i(K))$. 

\end{definition}

Let $K$ be a global field and let $G$ be a smooth connected linear algebraic $K$-group (any further requirement will be made explicit). Let $\Gm$ be a finitely generated Zariski dense subgroup of $G(K)$, and $A$ be a finite symmetric generating multiset of $\Gm$, with $1\in A$.

If $\p$ is a nonzero prime of $\ca O_K$ (the ring of integers of $K$) of sufficiently large norm, the equations defining $G$ can be reduced modulo $\p$, to give a smooth connected linear algebraic group over $K(\p):=\ca O_K/\p$, that we still denote by $G$. Since $\Gm$ is finitely generated, if $\p$ is sufficiently large we get a well defined map $\Gm \to G(K(\p))$.

If $Q$ is a finite quotient of $\Gm$, we will consider the Cayley graph of $Q$ with respect to the multiset induced by $A$. We will say that $Q$ is a \textit{one-sided $\eps$-expander with respect to $A$} if the second eigenvalue of the normalized adjacency operator of the corresponding Cayley graph is at most $1-\eps$. 





\begin{theorem}
\label{t_general}
With notation as above, assume that there exists a set $\mathscr P$ of primes of $K$ of density one such that, for every distinct $\p, \q\in\mathscr P$, the morphism $\Gm \rightarrow G(k(\p))\times G(k(\q))$ is well defined and surjective. Let $Z$ be a thin set of $G(K)$ of complexity $M$.
\begin{enumerate}
\item \label{i_polynomial} There exists $C_1>0$ such that
\[
\P(\omega_n \in Z) \leq C_1 n^{-1/(10\,\dim \, G)},
\]
where $C_1$ depends on $G$, $K$, $|A|$, $M$.
\item \label{i_exponential} Assume furthermore that, for every distinct $\p, \q\in\mathscr P$, the finite groups $G(k(\p))\times G(k(\q))$ are one-sided $\eps$-expanders with respect to $A$. Then, 
there exists $C_2>0$ such that
\[
\P(\omega_n \in Z) \leq e^{-n/C_2},
\]
where $C_2$ depends on $G$, $K$, $|A|$, $M$, $\eps$.
\end{enumerate}
\end{theorem}


\subsection{Group large sieve method}
\label{subsec_large_sieve}

The proof of Theorem~\ref{t_general} is based on the ``group large sieve method'', introduced by Kowalski \cite[Chapter 7]{kowalski_book_large_sieve}; revised and reformulated in a neat way by Lubotzky--Meiri \cite{lubotzky_meiri_powers}.

We first recall the main ideas of this method. 
Let $\Gm$ be a finitely generated group, $A$ be a finite symmetric generating multiset of $\Gm$, with $1\in A$, and $Z$ be a subset of $\Gm$. The group sieve method provides conditions for which a random walk on $\mathrm{Cay}(\Gm,A)$ hits elements of $Z$ with probability tending to zero exponentially fast.

The method is based on looking at finite quotients of $\Gm$. In order to get exponentially fast decay, the crux is that the Cayley graphs induced on these quotients should be uniform expanders, so that the corresponding random walks gets equidistributed in logarithmic time.

If the quotients are not expanders, we do not have equidistribution in logarithmic time, but still, we have equidistribution in polynomial time. This is sufficient in order to show that a random walk on $\mathrm{Cay}(\Gm,A)$ hits elements of $Z$ with probability decaying as $n^{-c}$. We make this precise in Lemma~\ref{lemma_sieve_polynomial}, below, which is the ``polynomial'' analogue of \cite[Theorem B]{lubotzky_meiri_powers}; see also \cite[Lemma 7.5]{Breuillard_superstrong}. We first need a lemma.

\begin{lemma}
\label{l_adjacency operator}
Let $G$ be a finite group, and let $A$ be a finite symmetric generating multiset of $G$, with $1\in A$. Let $\omega_n$ denote the $n$-th step of the uniform random walk on $\Cay(G,A)$, starting at the identity. Then, for every $g\in G$,
\[
\left|\P(\omega_n =g) - \frac{1}{|G|}\right| \leq |G|^{1/2} \cdot \left(1-\frac{1}{|A||G|^2}\right)^n.
\]
\end{lemma}

\begin{proof}
Let $1=\pi_0 \geq \pi_1 \geq \cdots \geq \pi_{\ell-1}\geq -1$ be the eigenvalues of the normalized adjacency operator associated to $\Cay(G,A)$. Set $\pi_*:=\max\{\pi_1, |\pi_{\ell-1}|\}$. By \cite[Lemma 2]{diaconis_saloff_comparison}, we have that, for every $g\in G$,
\[
\left|\P(\omega_n =g) - \frac{1}{|G|}\right| \leq |G|^{1/2} \cdot \pi_*^n.
\]
Now, \cite[Lemma~1 and Corollary~1]{diaconis_saloff_comparison} show that
\[
\pi_* \leq 1- \frac{1}{|A||G|^2}.
\]
This concludes the proof.
\end{proof}

We recall another lemma, which is an application of Chebyshev's inequality.

\begin{lemma}
\label{pairwise_independence}
Let $A_i$, $i=1,\ldots, t$ be events in a probability space. Assume that
\begin{enumerate}
    \item $\P(A_i) \leq 1-\beta$ for every $i$.
    \item $\P(A_i\cap A_j) \leq \P(A_i)\P(A_j) + \delta$ for every $i\neq j$.
\end{enumerate}
Then,
\[\P\left(\bigcap_{i=1}^t A_i\right)\leq \frac{1}{\beta^2}\left(\delta + \frac{\beta}{t}\right).
\]
\end{lemma}

\begin{proof}
See \cite[Lemma 3.1]{lubotzky_meiri_powers}. (In fact, that lemma is stated with $1/t$ replacing $\beta/t$ in the conclusion, but the proof gives the above, slightly stronger, statement.)
\end{proof}

\begin{lemma} 
\label{lemma_sieve_polynomial}
Let $\Gamma$ be a finitely generated group, let $A$ be a finite symmetric generating multiset of $\Gamma$, with $1\in A$, and let $Z$ be a subset of $\Gm$. Let $N_1, \ldots, N_t$ be finite index normal subgroups of $\Gamma$. For every $i$, denote by $\pi_i:\Gm\rightarrow \Gm/N_i$ the natural projection. Assume that there exist positive constants $C$, $D$ and $0<\alpha <1$ such that the following conditions are satisfied.
\begin{enumerate}
    \item (CRT condition) $\Gm/(N_i\cap N_j) \cong \Gm/{N_i} \times \Gm/{N_j}$ for every $i \neq j$.
    \item (Quotients grow sub-polynomially) $|\pi_i(\Gm)| \leq C\cdot t^D$ for every $i$.
    \item  (Large sieve assumption) $|\pi_i(Z)| \leq (1-\alpha)|\pi_i(\Gm)|$ for every $i$. 
\end{enumerate}
Then, for every $n \geq 10|A|C^5 t^{5D}/\alpha$,
\[
\P(\omega_n \in Z) \leq \frac{3}{\alpha t}.
\]
In particular, if $t\asymp n^{1/(5D)}$, we get $\P(\omega_n \in Z) \ll n^{-1/(5D)}$, where the implied constant depends polynomially on $1/\alpha, |A|, C$.
\end{lemma}


\begin{proof}
For every $i\in\{1, \ldots, t\}$ and for every $y\in \pi_i(\Gm)$, by Lemma~\ref{l_adjacency operator} and by assumption (2) we have
\begin{align}
\label{eq_diaconis}
    \left|\P(\pi_i(\omega_n)=y) - \frac{1}{|\pi_i(\Gm)|}\right| &\leq |\pi_i(\Gm)|^{1/2} \cdot \left(1-\frac{1}{|A||\pi_i(\Gm)|^2}\right)^n \\
    &\leq C^{1/2}t^{D/2} \cdot \left( 1- \frac{1}{|A|C^2t^{2D}}\right)^n. \nonumber
\end{align}


By the same reasoning and by assumption (1), for every $i \neq j$ and $y_i\in \pi_i(\Gm)$, $y_j\in \pi_j(\Gm)$, we have
\begin{align}
\label{eq_diaconis_2}
    &\left|\P(\pi_i(\omega_n)=y_i, \pi_j(\omega_n)=y_j) - \frac{1}{|\pi_i(\Gm)||\pi_j(\Gm)|}\right| \\
    &\qquad \leq (|\pi_i(\Gm)||\pi_j(\Gm)|)^{1/2} \cdot \left(1-\frac{1}{|A|(|\pi_i(\Gm)||\pi_j(\Gm)|)^2}\right)^n
    \nonumber  \\
    &\qquad  \leq Ct^D \cdot \left( 1- \frac{1}{|A|C^4t^{4D}}\right)^n. \nonumber
\end{align}
Now, let $A_i$ denote the event ``$\pi_i(\omega_n) \in \pi_i(Z)$''. Applying a union bound and three times the triangle inequality, we deduce from \eqref{eq_diaconis} and \eqref{eq_diaconis_2} that
\[
\P(A_i\cap A_j) \leq \P(A_i)\P(A_j) + 3 C^3 t^{3D} \cdot \left( 1- \frac{1}{|A|C^4t^{4D}}\right)^n.
\]
Set now
\begin{align*}
    \delta &:=3 C^3 t^{3D} \cdot \left( 1- \frac{1}{|A|C^4t^{4D}}\right)^n,\\
    \beta &:= \alpha/2.
\end{align*}
For every $n \geq 10|A|C^5 t^{5D}/\alpha$, we have that $4\delta/\alpha^2 \leq 1/\alpha t$.
Applying Lemma~\ref{pairwise_independence} to the events $A_1, \ldots, A_t$, we get that
\[
\P(\omega_n \in Z) \leq \P\left(\bigcap_{i=1}^t A_i\right) \leq \frac{4}{\alpha^2}\left( \delta  + \frac{\alpha}{2t}\right) \leq \frac{3}{\alpha t}.
\]
This concludes the proof.
\end{proof}

We will now prove Theorem~\ref{t_general}.

\begin{proof}[Proof of Theorem~\ref{t_general}]

We will first prove item \eqref{i_polynomial}. Since the constant $C_1$ in the statement depends on the complexity $M$ of the thin set $Z$, we may assume that $Z=C(K)$ where $C$ is a proper closed subvariety of $G$, or $Z=\pi(V(K))$ for some separable dominant morphism $\pi\colon V\to G$, with $V$ quasi-projective, $\dim(V)=\dim(G)$, and $\deg(\pi) \geq 2$. The case where $Z=C(K)$ will be dealt in Lemma~\ref{l_proper_closed}, below, under weaker assumptions. Therefore, we focus here on the case $Z=\pi(V(K))$. We may assume that $V$ is geometrically irreducible, for otherwise $V(K)$ is contained in a proper closed subset of $V$, hence $Z$ is contained in a proper closed subset of $G$ and we are in the previous case.

Let $K_\pi$ be the algebraic closure of $K$ in the Galois closure of $K(V)/K(G)$. 
Let $\mathscr P$ be the set of primes given in the statement of Theorem~\ref{t_general}. Let $\mathscr P'$ be the set of primes $\p$ of $\mathscr P$ having the following property: 

\begin{itemize}
    \item[$\diamond$] $\p$ splits completely in $K_\pi$, and $\text N(\p)$ is sufficiently large, depending on $G$ and $M$, in such a way that 
    $|\pi(V(K(\p)))| \leq (1-\frac{1}{2M!}) \cdot |G(K(\p))|$. 
    (Here, we still use the letters $G$ and $V$ to denote the base change to $K(\p)$ of models over a suitable ring of $S$-integers of $K$.)
\end{itemize}
We can ensure this condition in view of the Lang--Weyl estimate \cite{lang_weyl} and in view of \cite[Theorem 3.6.2]{serre_topics_galois}. The latter is stated over number fields, but the proof works also over function fields if $\pi$ is separable.

We order the primes in $\mathscr P'$, in such a way that the norm of the primes is non-decreasing, and we get a sequence $(\p_i)_{i \in \N}$. Let $n$ be a positive integer; our aim is to bound $\P(\omega_n\in Z)$. We want to apply Lemma~\ref{lemma_sieve_polynomial} to the finite quotients $G(K(\p_i))$, $i=1, \ldots, t$, of $\Gm$, for some $t$ that will be chosen later. We now check that conditions (1), (2) and (3) of that lemma are satisfied. Condition (1) is the main assumption of Theorem~\ref{t_general}, so it is satisfied.




By Chebotarev Density Theorem, the density of primes of $K$ splitting completely in $K_\pi$ is $1/|K_\pi : K| \geq 1/M!$, hence so is the density of primes in $\mathscr P'$. Combining this with Landau Prime Ideal Theorem, 
we get $\text N(\p_i)\ll i\log i\ll i^2$, where the implied constant depends on $G$, $K$, $M$. 
Hence, by the Lang--Weyl estimate, we get $|G(K(\p_i))|\ll  i^{2\,\dim(G)}$. In particular, condition (2) of Lemma~\ref{lemma_sieve_polynomial} is satisfied, with $D=2\cdot\dim(G)$. Finally, condition (3) holds by the definition of $\mathscr P'$. At this point, we can choose any $t\asymp n^{1/(10\,\dim \, G)}$, and Theorem~\ref{t_general}\eqref{i_polynomial} follows from Lemma~\ref{lemma_sieve_polynomial}.

The proof of Theorem~\ref{t_general}\eqref{i_exponential} is similar, but instead of using Lemma~\ref{lemma_sieve_polynomial}, we use \cite[Theorem B]{lubotzky_meiri_powers}. 
\end{proof}



\subsection{Proof of Theorems~\ref{t_main_new} and~\ref{t_main_new_function_fields} for $G$ simply connected}
\label{sub_proof_simply_connected}

\begin{proof}[Proof of Theorem~\ref{t_main_new} for $G$ simply connected]
Let $G$ be a connected simply connected linear algebraic $K$-group, where $K$ is a number field, and let $\Gm$ be a finitely generated Zariski dense subgroup of $G(K)$. Let $\mathscr P$ be the set of nonzero primes of $\ca O_K$ having inertia degree one over $\Q$, and of sufficiently large norm so that $G$ can be reduced modulo $\p$ and we have a map $\Gm \to G(K(\p))$ for every $\p\in \mathscr P$. Note that $\mathscr P$ has density one among all primes. We recall two facts.

(i) The ``strong approximation theorem'' (see Nori \cite[\S~5]{nori_subgroups_gln}) implies that, possibly removing finitely many primes from $\mathscr P$, the morphism $\Gm \to G(K(\p))\times G(K(\q))$ is surjective for every $\p\neq \q \in \mathscr P$. (We refer to \cite[pp. 399--417]{lubotzky_segal} and \cite[Chapter 7]{platonov_rapinchuk} for thorough discussions on strong approximation.)

(ii) Assume, furthermore, that $G$ is semisimple. The ``superstrong approximation theorem'', proved by Salehi Golsefidy and Varj\'u \cite[Corollary 6]{Varju_expansion}, asserts that, possibly removing finitely many primes from $\mathscr P$, the quotients $G(k(\p))\times G(k(\q))$ of $\Gm$ are $\eps$-expanders for every $\p\neq \q \in \mathscr P$ and for some $\eps>0$ independent of $\p$. 

Now Theorem~\ref{t_main_new} for $G$ simply connected follows from (i), (ii), and Theorem~\ref{t_general}.
\end{proof}

\begin{proof}[Proof of Theorem~\ref{t_main_new_function_fields}]
Let $K$ be a global function field, let $G$ be connected and simply connected and let $\Gm=G(\ca O_S)$ be an $S$-arithmetic group, as in \S~\ref{sub_intro_function_fields}. Prasad \cite{prasad_strong_approximation} proved the strong approximation theorem for function fields. (The theorem is stated for simply connected $K$-simple groups, but a simply connected semisimple group is a direct product of such factors, and strong approximation behaves well with respect to direct products; see, e.g., \cite[Proposition 7.1]{platonov_rapinchuk}.) This implies that the assumptions of Theorem~\ref{t_general} are satisfied, with $\mathscr P$ denoting the set of all primes of sufficiently large norm (in terms of $G$ and $\Gm$). Then, Theorem~\ref{t_main_new_function_fields} follows from Theorem~\ref{t_general}\eqref{i_polynomial}.
\end{proof}

Recall that, in the proof of Theorem~\ref{t_general}, we only dealt with thin sets of the form $Z=\pi(V(K))$. We now close the gap, by proving Theorems~\ref{t_main_new} and~\ref{t_main_new_function_fields} in the case $Z=C(K)$ where $C$ is a proper closed subset of $G$. In particular, over number fields we do not require that $G$ is simply connected.
 The case $Z=C(K)$ is easier: It is sufficient to reduce modulo a single prime, rather than modulo pairs of primes, as in Lemma~\ref{lemma_sieve_polynomial}. We decided to postpone the proof here because we need the results on strong and superstrong approximation, which we recalled only in this subsection. 


\begin{lemma}
\label{l_proper_closed}
Theorems~\ref{t_main_new} and~\ref{t_main_new_function_fields} hold when $Z=C(K)$ and $C$ is a proper closed subset of $G$.
\end{lemma}

\begin{proof}
Let $\p$ be a prime of sufficiently large norm, to be determined later, and of inertia degree one over $\Q$ in the case where $K$ is a number field. Denote by $\pi_{\p}\colon \Gm \to G(K(\p))$ the reduction morphism. We note that the index of $\pi_{\p}(\Gm)$ in $G(K(\p))$ is bounded, independently of $\p$. Indeed, for function fields, this follows from strong approximation \cite{prasad_strong_approximation}; for number fields (where $G$ is not necessarily simply connected), see \cite[\S~5]{nori_subgroups_gln}.
Then, by the Lang--Weyl estimate, $|\pi_{\p}(\Gm)| \asymp \text N(\p)^{\dim(G)}$, and $|\pi_{\p}(Z\cap \Gm)| \ll \text N(\p)^{\dim(G)-1}$, so $|\pi_{\p}(Z\cap \Gm)|\ll |\pi_p(\Gm)|/\text N(\p)$. 
Now, reasoning as in \eqref{eq_diaconis} in the proof of Lemma~\ref{lemma_sieve_polynomial}, we get 
\begin{align*}
\P(\omega_n \in Z)&\leq \P(\pi_{\p}(\omega_n) \in \pi_{\p}(Z\cap \Gm)) \\
&\leq \frac{|\pi_{\p}(Z\cap \Gm)|}{|\pi_{\p}(\Gm)|} + |\pi_{\p}(Z\cap \Gm)||\pi_{\p}(\Gm)|^{1/2} \cdot \left(1-\frac{1}{|A||\pi_{\p}(\Gm)|^2}\right)^n.
\end{align*}
If we choose $\p$ of norm $\asymp n^{1/(3\,\dim\, G)}$, we get $\P(\omega_n\in Z) \ll n^{-1/(3\,\dim\, G)}$, which proves Theorem~\ref{t_main_new} for $G/\urad(G)$ semisimple, as well as Theorem~\ref{t_main_new_function_fields}. Assume now that $G$ is semisimple. Then, the quotients $\pi_{\p}(\Gm)$ are uniform expanders \cite{Varju_expansion}, hence (in the notation of the proof of Lemma~\ref{l_adjacency operator}) $\pi_*$ is bounded away from one, and we can choose $\p$ of norm $\asymp e^{n/C}$, so that $\P(\omega_n\in Z) \ll e^{-n/C}$. This concludes the proof of Theorem~\ref{t_main_new}. (The statement for $G$ semisimple and $Z=C(K)$ is known and can be found for instance in \cite[Proposition 2.7]{lubotzky_meiri_powers}.)
\end{proof}

Even in the case $Z=C(K)$, in general we cannot get exponential decay, as \eqref{r_intro_example} demonstrates. We conclude this section with some remarks on Theorem~\ref{t_main_new} and Theorem~\ref{t_main_new_function_fields}.
\begin{remark}
\label{r_no_assumptions}
In the proof of Theorem~\ref{t_main_new} for $G$ simply connected, we have not used any assumption on $Z$, beyond the fact that it is a thin set. Therefore, for simply connected groups $G$, in Theorem~\ref{t_main_new}  it is not necessary to assume that the varieties are normal and the morphisms are finite and ramified. (We recall that, in fact, $G$ does not admit absolutely irreducible nontrivial unramified covers.) 
The same remark applies to Corollary~\ref{c_main_irreducibility} and Corollary~\ref{c_main_new}, which will be proved in \S~\ref{sec_pullback_argument}.

\end{remark}

\begin{remark}
\label{r_lubotzky_zuk}
Let $K$ be a global function field and let $G$ and $\Gm$ be as in Theorem~\ref{t_main_new_function_fields}. There exists a finite place $\nu$ of $K$ such that, if $\nu\in S$, then the quotients $G(K(\p))\times G(K(\q))$ of $\Gm=G(\ca O_{S})$ are uniform expanders, for $\p\neq \q$ of sufficiently large norm (see Lubotzky and Zuk \cite[Theorem 4.4]{lubotzky_zuk_tau}). If such $\nu$ belongs to $S$ then, using Theorem~\ref{t_general}\eqref{i_exponential}, we deduce that we can get exponential decay in Theorem~\ref{t_main_new_function_fields}:
\[
\P(\omega_n\in Z) \leq e^{-n/C}.
\]
In fact, conjecturally Theorem 4.4 of \cite{lubotzky_zuk_tau} should hold for every $S$, see \cite[remarks after Theorem 4.4]{lubotzky_zuk_tau}, so that in Theorem~\ref{t_main_new_function_fields} one would get exponential decay for every $S$. The same remark applies to Corollaries~\ref{c_function_fields_galois} and~\ref{c_main_function_fields_2}, which will be proved in \S~\ref{sec_pullback_argument} and \S~\ref{sub_characteristic_polynomial}.

\end{remark}

\begin{remark}
\label{rem_function_fields}
In Theorem~\ref{t_main_new_function_fields}, we require that $\Gm=G(\ca O_S)$ is finitely generated (which, of course, is a necessary assumption for our method). We note that this holds in many cases. For instance, if $G$ is $K$-simple, then $G(\ca O_S)$ is finitely generated unless $\sum_{\nu\in S}\text{rank}(G(K_\nu))=1$ (see \cite[Remark 2.2]{lubotzky_weiss}).
\end{remark}

\begin{remark}
\label{rem_function_general}
The proof of Theorem~\ref{t_main_new_function_fields} works under more general assumptions. For instance, let $S$ be a finite set of places, and $L$ be simply connected, semisimple, such that $\prod_{\nu \in S} L_i(K_\nu)$ is non-compact for every $K$-simple factor $L_i$ of $L$. Let $U$ be unipotent and $K$-split (i.e., it admits a composition series with factors isomorphic to $\mathbb A_K^1$). Assume that $L$ acts on $U$, and set $G=U\rtimes L$ and $\Gm=G(\ca O_S)$. Assume that $\Gm$ is finitely generated.

Then, Theorem~\ref{t_main_new_function_fields} holds for $G$ and $\Gm$. Indeed, $U$ is isomorphic to some $\mathbb A_K^m$ as a $K$-variety (see \cite[Remark A.3]{kambayashi}), and so satisfies strong approximation with respect to $S$. Now, as a $K$-variety $G\cong \mathbb A_K^m \times L$, therefore $G$ satisfies strong approximation with respect to $S$ and the claim follows from Theorem~\ref{t_general}(\ref{i_exponential}). 

\end{remark}

\section{General case: a pullback argument}
\label{sec_pullback_argument}
In this section we will prove Theorem~\ref{t_main_new} and Corollaries~\ref{c_main_irreducibility}, \ref{c_main_new} and \ref{c_function_fields_galois}.

We assume first that $K$ is a number field, and $G$ is a connected linear algebraic $K$-group, such that $G/\urad(G)$ is either trivial or semisimple. In the first case, $G$ is unipotent, hence simply connected. 
This case was considered in \S~\ref{sub_proof_simply_connected}, hence we may assume that $G/\urad(G)$ is semisimple. Then, $G$ admits a universal cover (\cite[Theorem 18.25 and Remark 18.27]{milne_algebraic_groups}),
that is, a connected simply connected linear algebraic $K$-group $\wt G$, equipped with an isogeny (i.e., surjective homomorphism with finite central kernel) $p\colon \wt G\to G$. We fix this and other notation, as follows.

\begin{notationlist}
\label{notation_2}
\begin{itemize}
\item[$\diamond$] Notation list~\ref{notation_1} holds.
\item[$\diamond$] $\thickbar K$ is an algebraic closure of $K$.
\item[$\diamond$] $G$ is such that $G/\urad(G)$ is semisimple.
    \item[$\diamond$] $\wt G$ is the universal cover of $G$, with isogeny $p\colon \wt G \to G$.
    \item[$\diamond$] $E$ is a finite extension of $K$ such that $ p^{-1}(\Gm)\subseteq \wt G(E)$.
    \item[$\diamond$] $\wt \Gm:= p^{-1}(\Gm)$.
    \item[$\diamond$] $\wt A :=p^{-1}(A)$, a finite symmetric generating multiset of the Zariski dense subgroup $\wt \Gamma$ of $\wt G_E(E)$.
    \item[$\diamond$] $\wt{\omega_n}$ is the $n$-th step of the random walk on $\Cay(\wt \Gm,\wt A)$, starting at the identity.
\end{itemize}
\end{notationlist}

(In the definition of $\wt \Gm$ and $\wt A$, $p$ denotes the morphism $\wt G(\thickbar K)\to G(\thickbar K)$ induced on $\thickbar K$-points.) Now, let $V$ be a  geometrically irreducible $K$-variety, with $\dim(V)=\dim(G)$, and let $\pi: V \rightarrow G$ be a dominant morphism. Base change to $E$, and get a pullback diagram
\begin{equation}
\label{diagram}
 \begin{tikzcd}
\wt V:=V_E\times_{G_E} \wt G_E \arrow{r}{\pi'} \arrow[swap]{d}{p'} & \wt G_E \arrow{d}{p} \\%
V_E \arrow{r}{\pi}& G_E
\end{tikzcd}
\end{equation}

\begin{lemma}
\label{l_pullback}
Let $Z=\pi(V(K))$. Under Notation list~\ref{notation_2}, one of the following holds:
\begin{itemize}
    \item[(i)] There exists $C>0$ such that $\P(\omega_n \in Z)\leq \Xi(C,n)$.
    \item[(ii)] There exists an irreducible component $W$ of $\wt V$ such that $\pi'|_W\colon W\to \wt G_E$ has degree one. In particular, there exists a rational map $\theta\colon \wt G_E \dashrightarrow V_E$ such that $\pi\circ \theta = p$. 
\end{itemize}
\end{lemma}

\begin{proof}
In this proof, we will denote by $p,p',\pi,\pi'$ the morphisms induced on $\thickbar K$-points. 
Note, first, that $\P(\omega_n \in Z)=\P(\omega_n \in Z \cap \Gm) = \P(\wt{\omega_n}\in p^{-1}(Z\cap \Gm))$. Set now $Y:=V(K)\cap \pi^{-1}(\Gm)$. It follows from \eqref{diagram} that $\pi'(p'^{-1}(Y))=p^{-1}(Z\cap \Gm)$, from which $\P(\omega_n \in Z)=\P(\wt{\omega_n}\in \pi'(p'^{-1}(Y)))$. Moreover,
$p'^{-1}(Y)$ is a subset of $\wt V(E)$. Assume that (ii) does not hold; then $\pi'(p'^{-1}(Y))$ is a thin set of $\wt G_E(E)$. Since the complexity of $\pi'(p'^{-1}(Y))$ is bounded in terms of $G$ and the complexity of $Z$, (i) follows from Theorem~\ref{t_main_new} for simply connected groups, proved in \S~\ref{sub_proof_simply_connected} (see also Remark~\ref{r_no_assumptions}).  
This concludes the proof of the lemma.
\end{proof}

At this point, we prove Theorem~\ref{t_main_new}.

\begin{proof}[Proof of Theorem~\ref{t_main_new}]
Since the constant $C$ in the statement depends on the complexity of $Z$, we may assume that either $Z=X(K)$ where $X$ is a proper closed subvariety of $G$, or $Z=\pi(V(K))$ for some $\pi\colon V\to G$, where $V$ is a normal geometrically irreducible $K$-variety, and $\pi$ is finite, surjective and ramified. The case $Z=X(K)$ follows from Lemma~\ref{l_proper_closed}, hence we assume $Z=\pi(V(K))$.

If conclusion (i) of Lemma~\ref{l_pullback} holds, we are done. Assume then that conclusion (ii) holds, seeking for a contradiction. Since the finiteness of a morphism is stable under base change and restriction to closed subvarieties, we have that $\pi'|_W\colon W\to \wt G_E$ is finite and surjective.
Since $\wt G_E$ is normal, Zariski's main theorem
implies that $\pi'|_W$ is an isomorphism. Set $f:=p'\circ (\pi'|_W)^{-1}\colon \wt G_E \to V_E$. Then $f$ is a finite morphism, and $\pi\circ f=p$ is unramified. Since $V_E$ is normal, 
it follows that $\pi$ is unramified. This contradicts the choice of $\pi$, and the proof is concluded.
\end{proof}

\begin{remark}
\label{r_necessary_ramified}

(i) In Theorem~\ref{t_main_new}, the assumption that the morphisms are ramified is necessary. For instance, assume that $\pi\colon G_1\to G$ is an isogeny of algebraic $K$-groups, of degree at least two, and let $\Gm\leq G(K)$ be finitely generated and Zariski dense. Then, still denoting by $\pi$ the map on $\thickbar K$-points $G_1(\thickbar K)\to G(\thickbar K)$, we have that $Z:=\pi(\pi^{-1}(\Gm)\cap G_1(K))$ is a finite index subgroup of $\Gm$ (see, e.g., \cite[Lemma 16.4.14]{lubotzky_segal}). It is easy, then, to see that a random walk on a Cayley graph of $\Gm$ hits elements of $Z$ with positive probability, so that Theorem~\ref{t_main_new} fails.

(ii) The examples given in (i) cover essentially all the cases excluded in Theorem~\ref{t_main_new}. Indeed, in Theorem~\ref{t_main_new}, we can rephrase the assumption that the morphism $\pi_i\colon V_i\to G$ is ramified, by requiring that the morphism $\pi_i\colon: (V_i)_E\to G_E$ is not an isogeny of algebraic groups; that is, there does not exist an isomorphism of $E$-varieties $\phi_i\colon (V_i)_E \to G_i$, where $G_i$ is an algebraic $E$-group, and an isogeny $f_i\colon G_i \to G_E$ such that $f_i\circ \phi_i=\pi_i$. 


\end{remark}

\begin{remark}
\label{r_not_normal}
In Theorem~\ref{t_main_new}, the assumption that the varieties are normal is not crucial. Some modification in the statement, however, is needed, because ramification can be resolved by the normalization map. See Lemma~\ref{l_pullback} for a reformulation.
\end{remark}

We now deduce Corollaries~\ref{c_main_new} and~\ref{c_function_fields_galois} on generically Galois covers, and Corollary \ref{c_main_irreducibility} on irreducibility of fibres.

\begin{proof}[Proof of Corollaries~\ref{c_main_new} and~\ref{c_function_fields_galois}]
We first prove Corollary~\ref{c_main_new}, so that $K$ is a number field. Let $V$ be a normal irreducible $K$-variety,  and $\pi\colon V\to G$ be finite, surjective and generically Galois, with group $\ca G$. We may assume $\deg(\pi)\geq 2$, otherwise the statement is trivial. Let $\Omega$ be the set of maximal subgroups of $\ca G$. For $\ca H\in \Omega$, set $V_{\ca H}:=V/\ca H$, a normal variety,
and denote by $\pi_{\ca H}\colon V_{\ca H}\to G$ the induced finite morphism.
Set $Z:=\cup_{\ca H\in \Omega} \pi_{\ca H}(V_{\ca H}(K))$, and let $U$ be an open dense subset of $G$ such that the cover $\pi^{-1}(U)\to U$ is \'etale. The set of elements $g$ of $U(K)$ whose specialized Galois group $\ca G_g$ is not generic is contained in $Z$ (see, e.g., \cite[Proposition 3.3.1]{serre_topics_galois}). In view of Lemma~\ref{l_proper_closed}, it is then sufficient to bound $\P(\omega_n\in Z)$.

Note, now, that if $V_{\ca H}$ is not geometrically irreducible, then $V_{\ca H}(K) = \varnothing$, and so can be removed from the definition of $Z$. On the other hand, if $V_{\ca H}$ if geometrically irreducible, then $\pi_{\ca H}$ is ramified, by the assumptions of Corollary~\ref{c_main_new}. In particular, $Z$ is a thin subset of $G(K)$ that satisfies the assumptions of Theorem~\ref{t_main_new}; so Theorem~\ref{t_main_new} holds for $Z$ and Corollary~\ref{c_main_new} follows.

The proof of Corollary \ref{c_function_fields_galois} is identical, using Theorems~\ref{t_main_new_function_fields} (which has no ramification assumption) instead of Theorem~\ref{t_main_new}.
\end{proof}

\begin{proof}[Proof of Corollary \ref{c_main_irreducibility}] We are under Notation list~\ref{notation_2}. Let $V$ be a normal irreducible $K$-variety, and $\pi\colon V\to G$ be a finite surjective morphism. Let 
$L:=\thickbar K \cap K(V)$, 
so that $K(G)\subseteq LK(G)\subseteq K(V)$ and the morphism $\pi$ factors as $\pi= f\circ \phi$, where $\phi\colon V \to G_L$ and $f\colon G_L \to G$ is the base change morphism. Then, we may regard $V$ as an absolutely irreducible $L$-variety. 
Up to replacing $E$ by $EL$, we may assume that $E$ contains $L$. Set $V_E:=V \times_L E$. Still denoting by $\phi$ and $f$ the induced morphisms $V_E\to G_E$ and $G_E \to G$, respectively, and by $\pi$ their composition, we have a diagram
\[
 \begin{tikzcd}
\wt V:=V_E\times_{G_E} \wt G_E \arrow{r}{\phi'} \arrow[swap]{d}{p'} & \wt G_E \arrow{d}{p} \\%
V_E \arrow[bend right]{rr}{\pi}\arrow{r}{\phi}& G_E\arrow{r}{f} &G
\end{tikzcd}
\]
By assumption and by the choice of $L$, every nontrivial subcover of $\phi$ is ramified. Since $p$ is unramified, it follows that $\wt V$ is irreducible. Let $\widehat V$ be the normalization of $\wt G_E$ in the Galois closure of $E(\wt V)/E(\wt G_E)$. We get a generically Galois morphism $\psi \colon \widehat V \to \wt G_E$, whose group we denote by $\ca G$. Then
\begin{align*}
&\P(\pi^{-1}(\omega_n) \mbox{ is not integral}) \\
&\qquad=\P(\phi^{-1}(\omega_n) \mbox{ is not integral})\\
&\qquad \leq \P(\phi'^{-1}(\wt \omega_n) \mbox{ is not integral}) \\
&\qquad \leq \P(\ca G_{\wt \omega_n} \neq \ca G) \\
&\qquad \leq \Xi(C,n),
\end{align*}
where the last inequality follows from Corollary \ref{c_main_new} (see also Remark \ref{r_no_assumptions}). This concludes the proof.
\end{proof}

\section{Applications}

In this section we prove Corollaries~\ref{c_main_char_pol},~\ref{c_main_function_fields_2} and~\ref{c_main_fixed_points}.

\subsection{Corollaries~\ref{c_main_char_pol} and~\ref{c_main_function_fields_2}: Characteristic polynomial}
\label{sub_characteristic_polynomial}
Let $K$ be a global field, and let $G$ be as in Corollary~\ref{c_main_char_pol} or~\ref{c_main_function_fields_2}.

We will use that $G$ has the HP over $K$, which, of course, follows from Theorems~\ref{t_main_new} and~\ref{t_main_new_function_fields}. (In fact, this is known for more general $G$ and $K$, see \cite{baysoroker_fehn_petersen, colliotthelene_sansuc}.)

Fix an embedding $\rho\colon G \hookrightarrow \GL_N$. 
Let $F$ be the largest separable extension of $K(G)$ contained in the splitting field of $\chi_{G,\rho}$ over $K(G)$.
Let $V$ be the normalization of $G$ in $F$, and let $\pi\colon V\to G$ be the corresponding finite surjective morphism, which is generically Galois with group $\ca G:=\Gal(F/K(G))$.

We now recall some preliminary facts from \cite[\S~2]{kowalski_jouve_zywina} or \cite{prasad2003existence}. For an element $g\in G(K)$, let $K_g$ the splitting field of $\det(X-g^\rho)$ over $K$. If $g$ is semisimple or $K$ is a number field, then $K_g$ is a Galois extension of $K$, which is independent of the choice of the embedding $\rho$. 
For every $g\in G(K)$, let $\ca G_g$ be the specialized group of $\pi$ at $g$. The following holds:
\begin{itemize}
\item[$(\star)$] There exists an open dense subset $U_0$ of $G$ such that, for every $g\in U_0(K)$, $K_g/K$ is Galois and $\ca G_g \cong\Gal(K_g/K)$.
\end{itemize}

Now, if $K$ is a number field, $K_g = K_{g^\phi}$, where $\phi\colon G \to G/\urad(G)$ (see \cite[Lemma 2.3]{kowalski_jouve_zywina}). In particular, if $K$ is a number field we may replace $G$ by $G^\phi$, $\Gm$ by $\Gm^\phi$, and assume that $G$ is semisimple. Note that this holds by assumption if $K$ is a function field.
\begin{itemize}

\item[$(\star\star)$] \label{erro} From now on, therefore, $G$ is a semisimple group.
\end{itemize}


Let $K^s$ be a fixed separable closure of $K$. For a maximal $K$-torus $T$ of $G$, we let $X(T) := \Hom(T_{K^s}, (\mathbb G_m)_{K^s})$, the character group of $T$, and let $K_T$ be the splitting field of $T$, a Galois extension of $K$ (\cite[Theorem 12.18]{milne_algebraic_groups}). Let $W(G,T)$ denote the $K^s$-points of the group scheme $\text N_G(T)/T$. Then $W(G,T)$ is a finite group, called the absolute Weyl group of $G$ with respect to $T$. Now, $W(G,T)$ and $\Gal(K_T/K)$ can be viewed as subgroups of $\Aut(X(T))$; define, then, $\Pi(G,T):=\gen{W(G,T), \Gal(K_T/K)} = W(G,T)\Gal(K_T/K)\leq \Aut(X(T))$. This is a finite group and, if $T$ is $K$-split, $\Pi(G,T)=W(G,T)$. It turns out that the isomorphism type of this group is independent of the choice of $T$, so that we can define $\Pi(G):=\Pi(G,T)$ for any choice of $T$.

It follows from a result of Prasad and Rapinchuk \cite{prasad2003existence}, and from the fact that $G$ has the HP over $K$, that $\Pi(G)$ is the generic Galois group of $\pi$:


\begin{lemma}
\label{l_generic_galois_group}
$\ca G\cong \Pi(G)$.
\end{lemma}
\begin{proof}
We divide the proof of the lemma in two parts.

(i) We show that $\ca G$ is isomorphic to a subgroup of $\Pi(G)$.

There exists an open dense subset $U$ of $G$ such that, for every $g\in U(K)$, $g$ is regular semisimple and, denoting by $T_g$ the maximal torus of $g$, $K_g = K_{T_g}$, so that $\Gal(K_g/K)$ is isomorphic to a subgroup of $\Pi(G)$ (see, e.g.,  \cite[Lemma 2.4]{kowalski_jouve_zywina}, which is stated over number field but does not use this). Let $U_0$ be as in $(\star)$. By the HP, we know that $\ca G_g\cong \ca G$ for at least one $g\in U\cap U_0(K)$, and in particular, $\ca G$ is isomorphic to a subgroup of $\Pi(G)$.

(ii) We show that $\ca G\cong \Pi(G)$.

In \cite[Theorem 3 and Remark 6]{prasad2003existence} (see also \cite[p. 242]{prasad2014generic}), it was shown that, if $G$ is semisimple, there exists a Zariski dense subset $U'$ of $G(K)$ such that, for every $g\in U'$, $g$ is regular semisimple and, with notation as in (i), $\Gal(K_{T_g}/K) = \Pi(G,T)$.  In particular, taking $g\in U'\cap U\cap U_0$, we get that $\ca G_g \cong \Gal(K_g/K) = \Gal(K_{T_g}/K) = \Pi(G,T)\cong \Pi(G)$. Hence $\Pi(G)$, as the Galois group of a specialization, is a sub-quotient of $\ca G$. From (i), we deduce that $\ca G\cong \Pi(G)$.
\end{proof}

\begin{remark}
In fact, we used only a  weak form of \cite{prasad2003existence}, namely, the existence of \textit{one} element $g\in G(K)$ such that $\Gal(K_g/K)\cong \Pi(G)$. This already implies that $\Pi(G)$ is the generic Galois group of $\pi$. In particular, by the HP, it follows automatically that the set of elements $g\in G(K)$ such that $\Gal(K_g/K)\cong \Pi(G)$ is Zariski dense in $G$.
\end{remark}

\begin{proof}[Proof of Corollaries~\ref{c_main_char_pol} and ~\ref{c_main_function_fields_2}]

By $(\star\star)$, we may assume that $G$ is semisimple. Let $Z$ be the set of elements $g$ of $G(K)$ such that $\det(X-g^\rho)$ has an inseparable irreducible factor or $\Gal(\det(X-g^\rho)/K)$ is not isomorphic to $\Pi(G)$. We want to show that $\P(\omega_n \in Z)\leq e^{-n/C}$. By Lemma~\ref{l_proper_closed}, it is sufficient to show that $\P(\omega_n \in Z')\leq e^{-n/C}$, where $Z':=Z\cap U\cap U_0$, and $U_0$ and $U$ were defined in $(\star)$ and in the proof of Lemma~\ref{l_generic_galois_group}, respectively. In particular, for $g\in Z'$, $g$ is regular semisimple, with maximal torus $T_g$, $K_g=K_{T_g}$, and $\ca G_g\cong \Gal(K_g/K)$.

Assume first that $G$ is simply connected. Then the statement follows from Corollary~\ref{c_main_new} (if $K$ is a number field) or Corollary~\ref{c_function_fields_galois} (if $K$ is a function field), applied to the cover $\pi\colon V \to G$. In particular, from now on we may assume that $G$ is not simply connected; hence, $K$ is a number field. We use now Notation list~\ref{notation_2}.

We have that $\P(\omega_n \in Z') = \P(\wt{\omega_n}\in p^{-1}(Z' \cap \Gm)$), where $p^{-1}(Z'\cap \Gm)\leq \wt G_E(E)$. Now let $g\in Z'$ and, for ease of notation, set $T:=T_g$. Then $\Gal(K_{T}/K) \nsupseteq W(G,T)$. Let $\wt g \in p^{-1}(Z'\cap \Gm)$; note that $\wt g$ is regular semisimple, and denote by $\wt T$ the maximal $E$-torus of $\wt G_E$ containing $\wt g$. Then $p(\wt T) = T_{E}$ and $E_{\wt T} = EK_T$, where $E_{\wt T}$ denotes the splitting field of the torus $\wt T$. We see that $\Gal(E_{\wt T}/E) \nsupseteq W(\wt G_E,\wt T)\cong W(G,T)$, from which $\Gal(E_{\wt g}/E)\not\cong \Pi(\wt G_E)$. Then, by the simply connected case, we deduce that $\P(\wt{\omega_n}\in p^{-1}(Z' \cap \Gm)) \leq e^{-n/C}$, which concludes the proof.
\end{proof}

\subsection{Corollary~\ref{c_main_fixed_points}: Fixed points}

\begin{proof}[Proof of Corollary~\ref{c_main_fixed_points}] We use Notation list~\ref{notation_2}. By assumption, $p^{-1}(\Gm)\leq \wt G(K)$, hence we may take $E=K$.

Note, first, that the set of elements of $G(K)$ fixing a point on $X(K)$ is $\pi(V(K))$. Let $W$ be an irreducible component of $V$. By assumption, there exists a Zariski dense subset of $G(\thickbar K)$ consisting of elements fixing finitely many points on $X(\thickbar K)$. This implies that, if $\pi|_W$ is dominant, then $\pi|_W$ has finite degree. We can, therefore, apply Lemma~\ref{l_proper_closed} or Lemma~\ref{l_pullback} to each irreducible component $W$ and to the set $Z=\pi(W(K))$. 

We deduce that,
if conclusion (ii) of Corollary~\ref{c_main_fixed_points} does not hold, then there exists a rational map $\theta': \wt G \dashrightarrow V$ such that $\pi\circ \theta' = p$.

At this point, \cite[Proposition 5.4]{corvaja2007rational} shows that $\theta'$ is constant on the fibres of $p$. In particular, there exists a rational map $\theta''\colon G \dashrightarrow V$ such that $\pi\circ \theta'' = \text{id}$. Setting $\theta:=\pi_X\circ \theta''$, where $\pi_X\colon V\to X$ is the natural projection, proves the first part of item (i) in the statement. Assume now that $X$ is projective; let $U$ be the domain of $\theta$ and let $g\in G(K)\setminus U$. Let $C$ be a smooth irreducible $K$-curve passing through $g$ and having non-empty intersection with $U$. Since $X$ is projective, the rational map $\theta|_C\colon C \dashrightarrow X$ can be extended to a morphism  $\thickbar \theta\colon C \to X$. Then $\thickbar \theta (g) g = \thickbar \theta (g)$, which proves the last part of item (i). (This argument can be found in \cite[p. 594]{corvaja2007rational}.) This concludes the proof.
\end{proof}

An interesting special case of Corollary~\ref{c_main_fixed_points} comes from the action of $G$ on flag varieties. Fix a morphism $\rho\colon G\to \GL_N$, and let $V$ be an $N$-dimensional $K$-vector space. Let $\mathbf a=(a_1, \ldots, a_t)$ be a sequence of integers, with $t\geq 1$ and $0<a_1<\cdots < a_t < N$. Then $G(K)$ acts naturally on the set of $\mathbf a$-flags of $V$
\[
0 < W_1 < \cdots < W_t < V,
\]
where $W_i$ is a $K$-subspace of $V$ of dimension $a_i$. Similarly, $G(\thickbar K)$ acts on the set of $\mathbf a$-flags of $\thickbar V:=V \otimes_K \thickbar K$.

\begin{corollary}
\label{c_flags}
Assume that, with notation as above, $p^{-1}(\Gm)\leq \wt G(K)$, and assume that there exists a Zariski dense subset of $G(\thickbar K)$ consisting of elements fixing finitely many $\mathbf a$-flags of $\thickbar V$. Then one of the following holds: 
\begin{enumerate}
    \item[(i)] Every element of $G(K)$ fixes an $\mathbf a$-flag of $V$.
    \item[(ii)] There exists $C>0$ such that
    \[
    \P(\mbox{$\omega_n$ fixes an $\mathbf a$-flag of $V$})\leq \Xi(C,n).
    \]
\end{enumerate}
\end{corollary}

\begin{proof}
The set of $\mathbf a$-flags of $V$ are the $K$-points of a projective $K$-variety $X$ which goes under the name of flag variety (see \cite[p. 146]{milne_algebraic_groups}). The morphism $\rho\colon G \to \GL_N$ equips $G$ with an action on $X$. The statement, then, follows from Corollary~\ref{c_main_fixed_points}.
\end{proof}

\bibliography{references}
\bibliographystyle{alpha}
\end{document}